\documentclass[letterpaper, 10 pt, conference]{ieeeconf}  

\IEEEoverridecommandlockouts                              
\usepackage{braket,amsfonts,color,makecell}
\usepackage{xcolor}

\usepackage{amsmath, amssymb, bbm, xspace}

\newcommand{\Real}{\ensuremath{\mathbb{R}}}

\DeclareMathOperator*{\st}{subject\;to}

\def\spose#1{\hbox to 0pt{#1\hss}}

\def\text #1{\hbox{\quad#1\quad}}

\def\nthinsp{\mskip -2   mu}

\def\superstar{^{\raise 0.5pt\hbox{$\nthinsp *$}}}
\def\SUPERSTAR{^{\raise 0.5pt\hbox{$*$}}}

\def\lamstarT {\lambda^{\raise 0.5pt\hbox{$\nthinsp *$}T}}

\def\Cscr{{\cal C}}

\def\Kscr{{\cal K}}

\def\Nscr{{\cal N}}

\def\Nscr{{\cal N}}

\def\Xscr{{\cal X}}

\def\hbar{\skew{4.2}\bar h}

		\def\bk1{{\rm 1\kern-.17em l}}
		\def\bkD{{\rm I\kern-.17em D}}
		\def\bkR{{\rm I\kern-.17em R}}
		\def\bkP{{\rm I\kern-.17em P}}
		\def\bkY{{\bf \kern-.17em Y}}
		\def\bkZ{{\bf \kern-.17em Z}}

		\def\beq{\begin{eqnarray}}
		\def\bc{\begin{center}}
		\def\be{\begin{enumerate}}
		\def\bi{\begin{itemize}}
		\def\bs{\begin{small}}
		\def\bS{\begin{slide}}
		\def\ec{\end{center}}
		\def\ee{\end{enumerate}}
		\def\ei{\end{itemize}}
		\def\es{\end{small}}
		\def\eS{\end{slide}}
		\def\eeq{\end{eqnarray}}
		\def\qed{\quad \vrule height7.5pt width4.17pt depth0pt}

	\def\cp2problem#1#2#3#4{\fbox
		 {\begin{tabular*}{0.9\textwidth}
			{@{}l@{\extracolsep{\fill}}l@{\extracolsep{6pt}}l@{\extracolsep{\fill}}c@{}}
				#1 & & $#4 $
			\end{tabular*}}}

		\renewcommand{\emph}[1]{\textbf{#1}}

		\def\bk1{{\rm 1\kern-.17em l}}
		\def\bkD{{\rm I\kern-.17em D}}
		\def\bkR{{\rm I\kern-.17em R}}
		\def\bkP{{\rm I\kern-.17em P}}
		
		\def\bkZ{{\bf{Z}}}

\newcommand {\beeq}[1]{\begin{equation}\label{#1}}
\newcommand {\eeeq}{\end{equation}}
\newcommand {\bea}{\begin{eqnarray}}
\newcommand {\eea}{\end{eqnarray}}

\def\texitem#1{\par\smallskip\noindent\hangindent 25pt
               \hbox to 25pt {\hss #1 ~}\ignorespaces}

\def\st{\mbox{subject to}}

\usepackage{amsthm,tcolorbox}

\usepackage{multirow}
\usepackage{framed}

\usepackage{shadethm}
\usepackage{epstopdf}
\usepackage[para,online,flushleft]{threeparttable}
\usepackage{algorithmic}
\usepackage{url}
\usepackage{verbatim}
\usepackage{amsmath} 
\usepackage{amssymb}  
\usepackage{amsfonts}
\usepackage{acronym}
\usepackage[normalem]{ulem}
\usepackage{graphicx}
\usepackage{epsfig}
\usepackage{cite}
\usepackage{algorithm}
	\newshadetheorem{thm}{theorem}
	\definecolor{shadethmcolor}{HTML}{EDF8FF}
	\definecolor{shaderulecolor}{HTML}{45CFFF}
	\definecolor{shaderulecolor}{gray}{0}
	\colorlet{shadecolor}{orange!15}

	\newshadetheorem{alg}{Algorithm}
	\definecolor{shadethmcolor}{HTML}{EDF8FF}
	\definecolor{shaderulecolor}{HTML}{45CFFF}
	\definecolor{shaderulecolor}{gray}{0}
	\colorlet{shadecolor}{orange!15}

	\newshadetheorem{prop}{proposition}
	\definecolor{shadethmcolor}{HTML}{EDF8FF}
	\definecolor{shaderulecolor}{HTML}{45CFFF}
\definecolor{shaderulecolor}{gray}{0}
 	\colorlet{shadecolor}{orange!15}

\newtheorem{theorem}{Theorem}
\usepackage{hyperref}
\newtheorem{assumption}{Assumption}
\newtheorem{definition}{Definition}

\newtheorem{lemma}{Lemma}
\newtheorem{proposition}{Proposition}
\newtheorem{remark}{Remark}
\def\bko{{\rm 1\kern-.17em l}}

\def\Nscr{{\mathcal N}}

\def\Kscr{{\mathcal K}}

\newcommand{\f}{\tilde{f}}

\def\be{\begin{enumerate}}
\def\ee{\end{enumerate}}
\def\Nscr{{\mathcal N}}
\def\st{\mbox{subject to}}

 \newcommand{\remove}[1]{}

\newcommand{\x}{{\mathbf{x}}}
\newcommand{\y}{{\mathbf{y}}}

\def\Real{\mathbb{R}}

\title{\bf Zeroth-order randomized block methods for constrained minimization of expectation-valued Lipschitz continuous functions}

\author{Uday V. Shanbhag \thanks{Industrial \& Manufacturing Eng.,
	Pennsylvania State University, \texttt{udaybag@psu.edu}; Shanbhag acknowledges the support from NSF CMMI-1538605 and DOE ARPA-E award DE-AR0001076. } \and Farzad Yousefian\thanks{School of Industrial Eng. \& Management, Oklahoma State University, \texttt{farzad.yousefian@okstate.edu}; 
		Yousefian acknowledges the support of the NSF through CAREER grant ECCS-1944500.		
		} }

\begin{document}

\maketitle
\thispagestyle{empty}
\pagestyle{empty}

\begin{abstract}  
    We consider the minimization of  an $L_0$-Lipschitz continuous and
    expectation-valued function, denoted by $f$ and defined as $f(\x)
    \triangleq \mathbb{E}[\tilde{f}(\x,\omega)]$, over a Cartesian product of
    closed and convex sets with a view towards obtaining both asymptotics as
    well as rate and complexity guarantees for computing an approximate
    stationary point (in a Clarke sense). We adopt a smoothing-based approach
    reliant on minimizing $f_{\eta}$ where $f_{\eta}(\x) \triangleq
    \mathbb{E}_{u}[f(\x+\eta u)]$, $u$ is a random variable defined on a unit
    sphere, and $\eta > 0$.  In fact, it is observed that a stationary point of
    the $\eta$-smoothed problem is a $2\eta$-stationary point for the original
    problem in the Clarke sense. In such a setting, we derive a suitable
    residual function that provides a metric for stationarity for the smoothed
    problem.    By leveraging a zeroth-order framework reliant on utilizing
    sampled function evaluations implemented in a block-structured regime, we
    make two sets of contributions for the sequence generated by the proposed
    scheme. (i) The residual function of the smoothed problem tends to zero
    almost surely along the generated sequence; (ii) To compute an $\x$ that
    ensures that the expected {norm of the} residual {of the $\eta$-smoothed
    problem} is within $\epsilon$ requires {no greater than}
    $\mathcal{O}(\tfrac{1}{{\eta} \epsilon^2})$ projection steps and
    $\mathcal{O}\left(\tfrac{1}{{\eta^2} \epsilon^4}\right)$ function
    evaluations. {These statements appear to be novel and there appear to be
    few results to contend with general nonsmooth, nonconvex, and stochastic
regimes via zeroth-order approaches.}

\end{abstract}

\section{Introduction}
 We consider the following stochastic optimization problem
\begin{align}\label{eqn:prob}
\begin{aligned}
    \min_{\x} & \quad f(\x) \triangleq  \mathbb{E}[{\f}(\x,\xi(\omega))] \\
    \st & \quad \x \in \Xscr\triangleq \textstyle\prod_{i=1}^b\Xscr_{i},
\end{aligned}
\end{align}
where $f:\mathbb{R}^n\to \mathbb{R}$ is a real-valued, nonsmooth, and nonconvex
function, $\Xscr_i \subseteq \mathbb{R}^{n_i} $ is a closed and convex set for 
$i=1,\ldots,b$ with $\sum_{i=1}^bn_i=n$, $\xi:\Omega \to \mathbb{R}^d$ denotes
a random variable associated with the probability space $(\Omega,
\mathcal{F},\mathbb{P})$. Throughout, we assume that $f$ is
$L_0$-Lipschitz continuous on the set $\Xscr$, i.e., there exists a scalar
$L_0>0$ such that for all $\x,\y \in \Xscr$ we have $|f(\x)-f(\y)| \leq
L_0\|\x-\y\|$. Further, at any $\x \in \mathbb{R}^n$,  $f(\x)\triangleq
\mathbb{E}[{\f}(\x,\omega)]$ where we refer to $f(\x,\xi(\omega))$ by
$f(\x,\omega)$.\\

While there is a significant body of literature on contending with nonsmooth
stochastic convex optimization problems~\cite{shapiro09lectures}, most
nonconvex generalizations are generally restricted to structured regimes where
the nonconvexity often emerges as  an expectation-valued
smooth function while the nonsmoothness arises in a deterministic form .
However, in many applications, $\tilde{f}(\bullet,\omega)$ may be
both nonconvex and nonsmooth and proximal stochastic gradient
schemes~\cite{ghadimi2016mini,lei2020asynchronous} cannot be directly adopted. We now discuss some relevant research in nonsmooth and nonconvex regimes. 

\medskip

\noindent {\em (a) Nonsmooth nonconvex optimization.}
In~\cite{burke02approximating}, Burke et al. disuss how gradient sampling
allow for approximating the Clarke subdifferential of a
function that is differentiable almost everywhere. 
A robust
{\em gradient sampling} scheme was subsequenly developed~\cite{burke05robust}
for functions that are continuously differentiable on an open subset $D
\subseteq \Real^n$; such functions need not be convex nor locally Lipschitz.
In~\cite{burke05robust}, the authors prove that a limit point of a subsequence
is an $\epsilon$-Clarke stationary point when $f$ is locally Lipschitz while
Kiwiel proved that every limit point is Clarke stationary with respect to $f$
without requiring compactness of level sets~\cite{kiwiel07convergence}. There
have also been efforts to develop statements in structured regimes where $f$ is either weakly convex~\cite{davis19stochastic,davis19proximally} or $f =
g+h$ and $h$ is smooth and possibly nonconvex while $g$ is convex, nonsmooth,
and proximable~\cite{xu2017globally,lei2020asynchronous}.

\smallskip    

\noindent {\em (b) Nonsmooth nonconvex stochastic optimization.} Much of the
efforts in the regime of stochastic nonconvex optimization have been restricted
to structured regimes where $f(\x) = h(\x) + g(\x)$, $h(\x) \triangleq
\mathbb{E}[F(\x,\omega)]$, $h$ is smooth and possibly nonconvex while $g$ is
closed, convex, and proper with an efficient proximal evaluation. In such
settings, proximal stochastic gradient techniques~\cite{ghadimi2016mini} and their
variance-reduced counterparts~\cite{ghadimi2016mini,lei2020asynchronous} were developed.    

\medskip    

\noindent {\em (c) Zeroth-order methods.} Deterministic~\cite{chen12smoothing}
and randomized smoothing~\cite{steklov1,ermoliev95minimization} have been the
basis for resolving a broad class of nonsmooth optimization
problems~\cite{mayne84,yousefian2012stochastic}. When the original function is
nonsmooth and nonconvex, Nesterov and Spokoiny ~\cite{nesterov17} examine
unconstrained nonsmooth and nonconvex optimization problems via Gaussian
smoothing. 

\smallskip

\noindent {\em Motivation.} Our work draws motivation from the recent work by
Zhang et al.~\cite{zhang20complexity} where the authors show that for a
suitable class of nonsmooth functions, computing an $\epsilon$-stationary point
is impossible in finite time. This negative result is a consequence of the
possibility that the gradient can change in an abrupt fashion, thereby
concealing a stationary point. To this end, they introduce a notion of
$(\delta,\epsilon)$-stationarity, a weakening of $\epsilon$-stationarity;
specifically, if $\x$ is $(\delta,\epsilon)$-stationary, then there exists a
convex combination of gradients in a $\delta$-neighborhood of  $\x$ that has
norm at most $\epsilon$. However, this does not mean that $\x$ is
$\delta-$close to an $\epsilon$-stationary point of $\x$ as noted by
Shamir~\cite{shamir21nearapproximatelystationary} since the convex hull might
contain a small vector without any of the vectors being necessarily small.
However, as Shamir observes that one needs to accept that the
$(\delta,\epsilon)$-stationarity notion may have such pathologies. Instead, an
alternative might lie in minimizing a smoothed function
$\mathbb{E}_{\bf u}[f(\x+\delta {\bf u})]$ where $u$ is a suitably defined random
variable. {This avenue allows for leveraging a richer set of techniques but may still be afflicted by similar challenges.}

\smallskip
\begin{tcolorbox}
{\em We consider a zeroth-order smoothing approach in constrained,
stochastic, and block-structured regimes with a view towards developing
finite-time and asymptotic guarantees.}
\end{tcolorbox}

\smallskip

\noindent {\em Contributions.} We develop a randomized zeroth-order framework
in regime where $f$ is block-structured and expectation-valued over $\Xscr$, a
Cartesian product of closed convex sets in which locally randomized
smoothing is carried out via spherical smoothing. This scheme leads to a
stochastic approximation framework in which the gradient estimator is
constructed via (sampled) function values. Succinctly, our main contributions
are captured as follows.      

\smallskip

\noindent (i) {\em Almost sure convergence guarantees.} {On applying the
randomized scheme to an $\eta$-smoothed problem, under suitable choices of the
steplength and mini-batch sequences, we show  that the norm of the residual function of the smoothed problem tends to zero almost surely. }

\smallskip 

\noindent (ii) {\em Rate and complexity guarantees.} It can be shown that the
expected squared residual (associated with the smoothed problem)
diminishes at the rate of $\mathcal{O}(1/k)$ in terms of projection steps on $\Xscr_i$, leading to a complexity of {$\mathcal{O}(\eta^{-2}\epsilon^{-4})$ in terms of sampled function evaluations.}

\smallskip {\bf Notation.} We use $\x$, $\x^T$, and $\|\x\|$ to denote a column vector,
its transpose, and its Euclidean norm, respectively. We use $\x^{(i)}$ to denote the $i$th block coordinate of vector
$\x=\left(\x^{(1)},\ldots,\x^{(b)}\right)$. Given a mapping $F:\mathbb{R}^n \to
\mathbb{R}^n$, $F_i$ denotes the $i$th block coordinate of $F$. We
define $f^*\triangleq \inf_{\x \in \Xscr} f(\x)$ and $f^*_\eta\triangleq
\inf_{\x \in \Xscr} f_\eta(\x)$, where $f_\eta(\x)$ denotes the smoothed
approximation of $f$. Given a continuous function, i.e., $f \in C^{0}$,
we write $f \in C^{0,0}(\Xscr)$ if $f$ is Lipschitz continuous on the
set $\Xscr$ with parameter $L_0$. Given a continuously differentiable
function, i.e., $f \in C^{1}$, we write $f \in C^{1,1}(\Xscr)$ if $\nabla f$ is
Lipschitz continuous on $\Xscr$ with parameter $L_1$. We
write a.s. for ``almost surely” and $\mathbb{E}[Z]$ denotes the
expectation of a random variable $Z$. {Given a scalar $u$, $[u]_+
\triangleq \max\{0,u\}$.}

\section{Stationarity and smoothing}
We first recap some concepts of Clarke's nonsmooth calculus~\cite{clarke98}
 that allow for providing stationarity conditions. We first define the directional derivative, a key object necessary in addressing nonsmooth optimization problems. 
 \begin{definition}[{\bf Directional derivatives and Clarke generalized gradient~\cite{clarke98}}] \em 
    The directional derivative of $h$ at $\x$ in a direction $v$ is defined as 
    \begin{align}
        h^{\circ}(\x,v) \triangleq  \limsup_{\y \to \x, t \downarrow 0} \left(\frac{h(\y+tv)-h(\y)}{t}\right).
\end{align}
The Clarke generalized gradient at $\x$ can then be defined as 
\begin{align}
    \partial h(\x) \triangleq  \left\{ \xi \in \Real^n \mid h^{\circ}(\x,v) \geq  \xi^T v, \quad \forall v \in \Real^n\right\}.
\end{align}
In other words, $h^{\circ}(\x,v) = \displaystyle \sup_{g \in \partial h(\x)}  g^Tv.$\qed
\end{definition}

\smallskip

 {If $h$ is $C^1$ at $\x$, the Clarke generalized gradient reduces to the standard gradient, i.e. $\partial h(\x) = \nabla_{\x} h(\x).$ If $\x$ is a local minimizer of $h$, then we have that $0 \in \partial h(\x)$. In fact, this claim can be extended to convex constrained regimes, i.e. if $\x$ is a local minimizer of 
        $\displaystyle \min_{\x \in \Xscr} \ h(\x)$,
then $\x$ satisfies $0 \in \partial h(\x) + \Nscr_{\Xscr}(\x)$, 
where $\Nscr_{\Xscr}(\x)$ denotes the normal cone of $\Xscr$ defined at
$\x$~\cite{clarke98}.} We now  review some properties of $\partial h(\x)$. In particular, if
$h$ is locally Lipschitz on an open set $\Cscr$ containing $\Xscr$, then $h$ is
differentiable almost everywhere on $\Cscr$ by Rademacher's
theorem~\cite{clarke98}. Suppose $\Cscr_h$ denotes the set of points where $h$
is not differentiable. We now provide some properties of the Clarke generalized
gradient.   \begin{proposition}[{\bf Properties of Clarke generalized
    gradients~\cite{clarke98}}] \em
     Suppose $h$ is $L_0$-Lipschitz continuous on $\Real^n$. Then the following hold.
\begin{enumerate}
    \item[(i)] $\partial h(\x)$ is a nonempty, convex, and compact  set and $\|g \| \leq {L_0}$ for any $g \in \partial h(\x)$. 
\item[(ii)] $h$ is differentiable almost everywhere. 
\item[(iii)] $\partial h(\x)$ is an upper semicontinuous map defined as 
    $$\hspace{-.5in}\partial h(\x) \triangleq \mbox{conv}\left\{g \mid g = \lim_{k \to \infty} \nabla_{\x} h(\x_k), \Cscr_h \not \owns \x_k \to \x\right\}.  
\qed $$
\end{enumerate}
\end{proposition}
\noindent We may also define the $\epsilon$-Clarke generalized gradient~\cite{goldstein77} 
\begin{align}
    \mbox{ as }
    \partial_{\epsilon} h(\x) \triangleq \mbox{conv}\left\{ \xi: \xi \in \partial h(\y), \|\x-\y\| \leq \epsilon\right\}.
\end{align}
When $f$ is nonsmooth and nonconvex on $\Xscr$, a closed and convex set, then 
to contend with nonsmoothness, we consider a locally randomized smoothing technique described as follows. Given a function $h:\mathbb{R}^n \to \mathbb{R}$ and a scalar $\eta>0$, a smoothed approximation of $h$ is denoted by $h_\eta$ defined as 
\begin{align}\label{def-smooth}
h_\eta (\x) \triangleq \mathbb{E}_{u \in \mathbb{B}}[h(\x+\eta u)].
\end{align} where $\mathbb{B}\triangleq \{u \in \mathbb{R}^n\mid \|u\|\leq 1\}$
{denotes} the unit ball and $u$ is uniformly distributed over $\mathbb{B}$.
{We now recall some properties} of spherical smoothing. Throughout, {$\mathbb{S}\triangleq \{v \in \mathbb{R}^n \mid
\|v\|=1\}$ denotes} the surface of $\mathbb{B}$ and $\Xscr_{\eta} \triangleq \Xscr+\eta \mathbb{B}$ represents the Minkowski sum of $\Xscr$ and $\eta \mathbb{B}$.

\begin{lemma}[{\bf Properties of spherical smoothing (cf. Lemma 1 in ~\cite{CSY2021MPEC})}] \label{lemma:props_local_smoothing}\em Suppose $h:\mathbb{R}^n \to \mathbb{R}$ is {a continuous function and {its smoothed counterpart $h_{\eta}$ is defined as \eqref{def-smooth}}, where $\eta>0$ is a given scalar}.  Then the following hold.

    \noindent (i) $h_{\eta}$ is {$C^1$} over $\Xscr$ {and} 
    \begin{align}
        \nabla_{\x} h_{\eta}(\x) = \left(\tfrac{n}{\eta}\right) \mathbb{E}_{v \in \eta \mathbb{S}} \left[h(\x+v) \tfrac{v}{\|v\|}\right] \quad \forall \x \in \Xscr. 
    \end{align}
    {Suppose $h \in C^{0,0}(\Xscr_{\eta})$ with parameter $L_0$.} For any $\x, \y \in \Xscr$, we have that (ii) -- (iv) hold. 

\begin{enumerate}
    \item[ (ii)] $| h_{\eta}(\x)-h_{\eta}(\y) | \leq L_0 \|\x-\y\|.$ 

    \item[ (iii)] $| h_{\eta}(\x) - h(\x)| \leq L_0 \eta.$ 

    \item[ (iv)] $\| \nabla_{\x} h_{\eta}(\x) -\nabla_{\x} h_{\eta}(\y)\| \leq 
\tfrac{L_0n}{\eta}\|\x-\y\|.$ 
        

 
    \item [(v)] {Suppose $h \in C^{0,0}(\Xscr_{\eta})$ with parameter $L_0$ {and} 
\begin{align*}
g_{\eta}(\x,v) \triangleq \left(\frac{n}{\eta}\right) \frac{(h(\x+v) - h(\x))v}{\|v\|}.
\end{align*}
 for $v \in \eta\mathbb{S}$.  Then, for} any $\x \in \Xscr$, {we have that} ${\mathbb{E}_{v \in \eta \mathbb{S}}}[\|{g_\eta}(\x,v)\|^2] \leq L_0^2 n^2$. \qed 
\end{enumerate}

\end{lemma}
{We restrict our attention to functions $f$ that are $L_0$-Lipschitz continuous over $\Xscr_{\eta_0} \triangleq \Xscr + \eta_0 \mathbb{B}$. Further, we assume that $\f(\x,\xi) - f(\x)$ admits suitable bias and moment properties.} 
We intend to develop schemes for computing approximate stationary points of \eqref{eqn:prob} by an iterative scheme. However, we first need to formalizing the relationship between the original problem and its smoothed counterpart.  This is provided in~\cite{CSY2021MPEC} by leveraging results from ~\cite{mayne84,mordukhovich09variational}.  
\begin{proposition}[{\bf Stationarity of $h_{\eta}$ and $\eta$-stationarity}]\label{prop_equiv} \em 
        Consider \eqref{eqn:prob} where $f$ is a locally Lipschitz continuous function and $\Xscr$ is a closed, convex, and bounded set in $\Real^n$. 


        \noindent (i) For any $\eta > 0$ and any $\x \in \Real^n$, $\nabla f_{\eta}(\x) \in \partial_{2\eta} f(\x)$. Furthermore, if $0 \not \in \partial f(\x)$, then there exists an $\eta$ such that $\nabla_{\x} f_{\tilde \eta} (\x) \neq 0$ for $\tilde{\eta} \in (0,\eta]$.    

        \noindent (ii) For any $\eta > 0$ and any $\x \in \Xscr$, 
        \begin{align}
        0 \in \nabla_{\x} f_{\eta}(\x) + \mathcal{N}_{\Xscr}(\x) \Rightarrow 0 \in \partial_{2\eta} f(\x)+ \mathcal{N}_{\Xscr}(\x). \hspace{-0.1in}\qed
        \end{align}
    \end{proposition}
    {Intuitively, this means that if $\x$ is a stationary point of the
        $\eta$-smoothed problem, then $\x$ is a $2\eta$-Clarke stationary point of the
        original problem.} {Next, we introduce a residual function that
        captures the departure from stationarity. Recall that when $h$ is a
        differentiable but possibly nonconvex function and $\Xscr$ is a closed
        and convex set, then  $\x$ is a stationary point of \eqref{eqn:prob}
        if and only if 
    $$ G_{\beta}(\x) \triangleq \beta\left( \x - \Pi_{\Xscr}\left[\x - \tfrac{1}{\beta} \nabla_{\x} f(\x) \right]\right) = 0. $$
When $f$ is not necessarily smooth as is the case in this paper, a residual of
the smoothed problem can be derived by replacing $\nabla_{\x} f(\x)$ by
$\nabla_{\x} f_{\eta}(\x)$. In particular, the residual $G_{\eta,\beta}$
denotes the stationarity residual with parameter $\beta$ of the $\eta$-smoothed
problem while $\tilde{G}_{\eta,\beta}$ represents its counterpart arising from
using a sampling-based estimate of $\nabla_{\x} f_{\eta}(\x)$. }  \begin{definition}[{\bf The residual
mapping}]\label{def:res_maps}
  {Suppose Assumption~\ref{ass-1} holds.} Given $\beta>0$ and a smoothing parameter $\eta>0$, for any $\x \in \mathbb{R}^n$ and $\tilde e \in \mathbb{R}^n$ an arbitrary given vector,
  let the residual mappings $G_{\eta,\beta}(\x)$ and $\tilde{G}_{\eta,\beta}(\x, {\tilde e})$ be defined as 
\begin{align}
    G_{\eta,\beta}(\x) &\triangleq \beta {\left(\x - \Pi_{\Xscr}\left[\x - \tfrac{1}{\beta} \nabla_x {f_{\eta}(\x)}\right]\right)} \mbox{ and }\\
    \hspace{-0.05in}\tilde{G}_{\eta,\beta}(\x, {\tilde e}) &\triangleq \beta\left( \x - \Pi_{\Xscr}\left[\x - \tfrac{1}{\beta} ({\nabla_x f_{\eta}(\x)}+\tilde{e}) \right]\right). \hspace{-0.1in} \qed
\end{align}
\end{definition}
{Unsurprisingly, one can derive a bound on $\tilde{G}_{\eta,\beta}(\x, {\tilde e})$ in terms of $G_{\eta,\beta}(\x)$ and $\tilde{e}$, as shown next.}
 \begin{lemma}[Lemma 10 in~\cite{CSY2021MPEC}]\label{lem:inexact_proj_2}\em
    Let Assumption~\ref{ass-1} hold. Then the following holds for any $\beta>0$, $\eta>0$, and $\x\in\mathbb{R}^n$.
\begin{align*}
    \|G_{\eta,\beta}(\x) \|^2 &  \leq  {2} \| \tilde{G}_{\eta,\beta}(\x, {\tilde e}) \|^2 + 2 \|\tilde{e}\|^2  .
\end{align*}
\end{lemma}


\section{A randomized zeroth-order algorithm} 
In this section, we provide the main assumptions, outline the proposed algorithm, and derive some preliminary results that will be utilized in the convergence analysis. 
\begin{assumption}[{\bf Problem properties}] 
\label{ass-1}\em
\noindent Consider problem \eqref{eqn:prob}.

\noindent (i) $f$ is $L_0$-Lipschitz on $\Xscr + \eta_0 \mathbb{B}$ for some $\eta_0 > 0$. 
    
\noindent (ii) $\Xscr_i \subseteq \Real^n$ is a nonempty, closed, and convex set for $i=1,\ldots,b$.

\noindent (iii) For all $\x \in \Xscr_{\eta}+ \eta_0 \mathbb{S}$ we have $\mathbb{E}[\f(\x,\omega) | \x] =f(\x)$.

\noindent (iv) For all $\x \in \Xscr_{\eta}+ \eta_0 \mathbb{S}$ we have $\mathbb{E}[\parallel \f(\x,\omega)- f(\x)\parallel^2 | \x] \leq \nu^2$ for some $\nu > 0$.
\end{assumption}
We introduce a variance-reduced randomized block-coordinate zeroth-order scheme presented by Algorithm \ref{algorithm:zo_nonconvex}.
 \begin{algorithm*}[hbt]
    \caption{\texttt{VR-RB-ZO}: Variance reduced projected randomized block zeroth-order method}\label{algorithm:zo_nonconvex}
 {   \begin{algorithmic}[1]
         \STATE\textbf{input:}  Given $\x_0 \in \Xscr$, stepsize $\gamma>0$, smoothing parameter $\eta>0$, mini-batch sequence $\{N_k\}$, and an integer $R$ randomly selected from $\{\lceil\lambda K\rceil ,\ldots,K\}$ using a {discrete uniform} distribution
    \FOR {$k=0,1,\ldots,{K}-1$}
    \STATE (i) Generate a random {mini-}{batch $v_{j,k} \in \eta \mathbb{S}$ for $j=1,\ldots,N_k$}  
    \STATE (ii) Generate a {discrete uniform} random variable $i_k$ from $\{1,\ldots,b\}$
    \STATE (iii) Compute a {mini-batch of zeroth-order gradient estimates for $j = 1, \ldots, N_k$} as 
                                             $$g_{\eta}(\x_k,v_{j,k},\omega_{j,k},i_k) :=\tfrac{n\left({\f}(\x_k+ v_{j,k},\omega_{j,k}) - {\f}(\x_k,\omega_{j,k})\right)v_{j,k}^{(i_k)}}{\left\|v_{j,k}\right\|\eta}$$ 

                                             \STATE (iv) Evaluate the mini-batch inexact zeroth-order gradient. $
               g_{\eta,N_k,i_k}(\x_k) = \frac{\sum_{j=1}^{N_k}  g_{\eta}(\x_k,v_{j,k},\omega_{j,k},i_k)}{N_k}$
               \STATE (v) {Update $\x_k$ as follows.}
$\x_{k+1}^{(i)}:=
		\begin{cases}
		\Pi_{\Xscr_{i}}\left[ \x_{k}^{(i)}-\gamma g_{\eta,N_k,i}(\x_k) \right],& \text{ if } i = i_k, \\
		\x_k^{(i)}, & \text{ if } i \neq i_k. 
		\end{cases}$
    \ENDFOR
        \STATE Return $\x_R$ 
\end{algorithmic}}
\end{algorithm*}
To this end, we define a zeroth-order gradient estimate of  $\f(\x_k, \omega_{j,k})$ as follows.
\begin{align*}
    g_{\eta}(\x_k,v_{j,k},\omega_{j,k}) \triangleq \frac{n({\f}(\x_k+ v_{j,k},\omega_{j,k}) - {\f}(\x_k,\omega_{j,k}))v_{j,k}}{\left\|v_{j,k}\right\|\eta} .
\end{align*}
{Intuitively, $g_{\eta}(\x_k,v_{j,k},\omega_{j,k})$ generates a gradient estimate by employing sampled function evaluations $\tilde{f}(\x_k,\omega_{j,k})$ and $\tilde{f}(\x_k+v_{j,k},\omega_{j,k})$; in short, the zeroth-order oracle, given an $\x_k$ and a perturbation vector $v_{j,k}$, produces two evaluations, i.e.   
$\tilde{f}(\x_k,\omega_{j,k})$ and $\tilde{f}(\x_k+v_{j,k},\omega_{j,k})$.}
We then employ a mini-batch approximation of this gradient estimate within a block coordinate (BC) structure. 
We formally define the stochastic errors emergent from the randomized BC scheme.
\begin{definition}\label{def:stoch_errors} 
For all $k\geq 0$ and $j=1,\ldots,N_k$ we define 
\begin{align}
&e_{j,k} \triangleq \nabla {\f}_\eta(\x_k,\omega_{j,k})-\nabla f_\eta(\x_k),\label{def:stoch_errors1}\\
&\theta_{j,k }\triangleq g_{\eta}(\x_k,v_{j,k},\omega_{j,k}) -\nabla {\f}_\eta(\x_k,\omega_{j,k}),\label{def:stoch_errors2} \\
&\delta_{j,k} \triangleq  b \mathbf{U}_{i_k}g_{\eta}(\x_k,v_{j,k},\omega_{j,k},i_k) -  g_{\eta}(\x_k,v_{j,k},\omega_{j,k}),\label{def:stoch_errors3}
\end{align}
where $e_k\triangleq \tfrac{\sum_{j=1}^{N_k}e_{j,k}}{N_k}$, $\theta_k\triangleq \tfrac{\sum_{j=1}^{N_k}\theta_{j,k}}{N_k}$, $\delta_k\triangleq \tfrac{\sum_{j=1}^{N_k}\delta_{j,k}}{N_k}$, and we define $\mathbf{U}_\ell \in \mathbb{R}^{n\times n_\ell}$ for $\ell \in \{1,\ldots,b\}$ such that $\left[\mathbf{U}_1, \ldots,\mathbf{U}_b\right] =\mathbf{I}_n $ and $\mathbf{I}_n$ denotes the $n\times n$ identity matrix. 
\end{definition}
The history of Algorithm \ref{algorithm:zo_nonconvex} at iteration $k$ is defined as
\begin{align*}
&\mathcal{F}_k \triangleq \cup_{t=0}^{k-1}\left(\{i_t\}\cup\left(\cup_{j=1}^{N_t}\{\omega_{j,t}, v_{j,t}\}\right)\right), \qquad \hbox{for } k \geq 1.
\end{align*} 
 We impose the following independence requirement on $\omega_{j,k}, v_{j,k}$, and $i_k$. Recall that both $v_{j,k}$ and $i_k$ are user-defined so this is a mild requirement.
\begin{assumption}[{\bf Independence} ]\label{assum:random_vars}
    Random samples $\omega_{j,k}$, $v_{j,k}$, and $i_k$ are generated independent of each other for all $k\geq 0$ and $1\leq j \leq N_k$.
\end{assumption}
We now analyze the bias and moment properties of three crucial error sequences. 
\begin{lemma}[{\bf Bias and moment properties of $e, \delta, $ and $\theta$}]\label{lem:stoch_error_var}\em
    Consider Definition \ref{def:stoch_errors}. Let Assumptions \ref{ass-1} and \ref{assum:random_vars} hold. The following hold almost surely for $k\geq 0$ and $N_k \geq 1$.

    \noindent (i) $\mathbb{E}[e_{j,k}\mid \mathcal{F}_k] =\mathbb{E}[\theta_{j,k}\mid \mathcal{F}_k]=\mathbb{E}[\delta_{j,k}\mid \mathcal{F}_k]=0$ {almost surely} for all $j=1,\ldots,N_k$.

    \noindent (ii) $\mathbb{E}[\|e_k\|^2\mid \mathcal{F}_k] \leq \tfrac{n^2\nu^2}{\eta^2N_k}$, $\mathbb{E}[\|\theta_k\|^2\mid \mathcal{F}_k] \leq \tfrac{L_0^2n^2}{N_k}$, and $\mathbb{E}[\|\delta_k\|^2\mid \mathcal{F}_k] \leq  \tfrac{3n^2(b-1)}{N_k}\left(\tfrac{\nu^2}{\eta^2}+ L_0^2+\left(\tfrac{ \hat f}{\eta}\right)^2 \right),$ {almost surely} where $ \hat f \triangleq  \sup_{\x \in \Xscr+\eta\mathbb{S}} f(x)$.
\end{lemma} 
\begin{proof}
    \noindent (i) To show that $\mathbb{E}[e_{j,k}\mid \mathcal{F}_k] =0$ a.s., we may write
\begin{align*}
&\mathbb{E}[e_{j,k}\mid \mathcal{F}_k] = \mathbb{E}[ \nabla \f_\eta(\x_k,\omega_{j,k})-\nabla f_\eta(\x_k)\mid \mathcal{F}_k]\\
&\stackrel{\tiny{\hbox{Lemma \ref{lemma:props_local_smoothing}}}}{=}  \tfrac{n}{\eta}\mathbb{E}\left[ \mathbb{E}\left[\tfrac{\left(\f(\x_k+v,\omega_{j,k})-f(\x_k+v)\right)v}{\|v\|} \mid \mathcal{F}_k \cup\{{\omega}_{j,k}\}\right]\mid \mathcal{F}_k\right]\\
&\stackrel{\tiny{\hbox{Assum. \ref{assum:random_vars}}}}{=}
\tfrac{n}{\eta}\mathbb{E}\left[\tfrac{\mathbb{E}\left[ \f(\x_k+v,{\omega})-f(\x_k+v)\mid \mathcal{F}_k\cup\{v\}\right]v}{\|v\|} \mid \mathcal{F}_k\right]\\
&\stackrel{\tiny{\hbox{Assum. \ref{ass-1}}}}{=}
\tfrac{n}{\eta}\mathbb{E}\left[ 0\times \tfrac{v}{\|v\|} \mid \mathcal{F}_k\right]=0,\qquad  \hbox{a.s.} 
\end{align*}
To show that $\mathbb{E}[\theta_{j,k}\mid \mathcal{F}_k] =0$ a.s., we write
 \begin{align*}
&\mathbb{E}[\theta_{j,k}\mid \mathcal{F}_k] = \mathbb{E}[ g_{\eta}(\x_k,v_{j,k},\omega_{j,k}) -\nabla \f_\eta(\x_k,\omega_{j,k})\mid \mathcal{F}_k] \\
& =\mathbb{E}\left[\tfrac{n\left(\f(\x_k+ v_{j,k},\omega_{j,k}) - \f(\x_k,\omega_{j,k})\right)v_{j,k}}{\left\|v_{j,k}\right\|\eta}\right. \\ 
&\left. -\tfrac{n}{\eta}\mathbb{E}\left[\tfrac{\f(\x_k+v,\omega_{j,k})v}{\|v\|} { \mid \mathcal{F}_k\cup\{\omega_{k,j}\}}\right]\mid \mathcal{F}_k\right]\\
&\stackrel{\tiny{\hbox{Assum. \ref{assum:random_vars}}}}{=}\mathbb{E}\left[\mathbb{E}\left[\tfrac{n\left(\f(\x_k+ v_{j,k},\omega_{j,k}) - \f(\x_k,\omega_{j,k})\right)v_{j,k}}{\left\|v_{j,k}\right\|\eta}\mid \mathcal{F}_k\cup\{v_{j,k}\}\right] \right.\\ 
&\left. -\tfrac{n}{\eta}{\mathbb{E}}\left[\f(\x_k+v,\omega_{j,k})\tfrac{v}{\|v\|}\mid {\mathcal{F}_k\cup\{v\}}\right] {\mid \mathcal{F}_k} \right]\\
& \stackrel{\tiny{\hbox{Assum. \ref{ass-1}}}}{=}{\mathbb{E}\left[\tfrac{n\left(f(\x_k+ v_{j,k}) - f(\x_k)\right)v_{j,k}}{\left\|v_{j,k}\right\|\eta} -\tfrac{n}{\eta}f(\x_k+v)\tfrac{v}{\|v\|}\mid \mathcal{F}_k\right]}\\
& =\tfrac{n}{\eta}\mathbb{E}_v\left[f(\x_k+v)\tfrac{v}{\|v\|}-f(\x_k+v)\tfrac{v}{\|v\|}\mid \x_k\right]=0, \qquad \hbox{ a.s.}  
\end{align*}
To show $\mathbb{E}[\delta_{j,k}\mid \mathcal{F}_k] =0$ a.s., we have {almost surely} 
 \begin{align*}
&\mathbb{E}[\delta_{j,k}\mid \mathcal{F}_k] \\
&= \mathbb{E}[  b \mathbf{U}_{i_k}g_{\eta}(\x_k,v_{j,k},\omega_{j,k},i_k) -  g_{\eta}(\x_k,v_{j,k},\omega_{j,k})\mid \mathcal{F}_k]\\
&\stackrel{\tiny{\hbox{Assum. \ref{assum:random_vars}}}}{=}
 \mathbb{E}\left[  \mathbb{E}\left[ b \mathbf{U}_{i_k}g_{\eta}(\x_k,v_{j,k},\omega_{j,k},i_k)\right. \right.\\ 
&\left. \left. -  g_{\eta}(\x_k,v_{j,k},\omega_{j,k})\mid \mathcal{F}_k\cup\{\omega_{j,k},v_{j,k}\}\right] { \mid \mathcal{F}_k} \right]\\
 &= \mathbb{E}\left[  \textstyle\sum_{i=1}^b bb^{-1} \mathbf{U}_{i}g_{\eta}(\x_k,v_{j,k},\omega_{j,k},i) -  g_{\eta}(\x_k,v_{j,k},\omega_{j,k})\right. \\
 &\left. \mid \mathcal{F}_k\right]= \mathbb{E}\left[ g_{\eta}(\x_k,v_{j,k},\omega_{j,k}) -  g_{\eta}(\x_k,v_{j,k},\omega_{j,k})\mid \mathcal{F}_k\right]=0.
\end{align*}

\noindent  (ii) To show the first inequality, we may bound the conditional second moment of $e_{j,k}$ as follows.
\begingroup
\allowdisplaybreaks
\begin{align*}
&\mathbb{E}[\|e_{j,k}\|^2\mid \mathcal{F}_k] = \mathbb{E}\left[\left\| \nabla \f_\eta(\x_k,\omega_{j,k})-\nabla f_\eta(\x_k)\right\|^2\mid \mathcal{F}_k\right]\\
&\stackrel{\tiny{\hbox{Lemma \ref{lemma:props_local_smoothing}}}}{=}
\tfrac{n^2}{\eta^2}\mathbb{E}\left[ \left\| \mathbb{E}\left[\f(\x_k+v_{j,k},\omega_{j,k})-f(\x_k+v_{j,k})\right.\right.\right.\\
&\left.\left.\left. { \mid \mathcal{F}_k \cup \{\omega_{k,j}\}}\right]\right\|^2\mid \mathcal{F}_k\right]
 \stackrel{\tiny{\hbox{Jensen's ineq.}}}{\leq}\tfrac{n^2}{\eta^2}\mathbb{E}\left[ \mathbb{E}\left[\left\| \f(\x_k+v_{j,k},\omega_{j,k})\right.\right.\right.\\
&\left.\left.\left.-f(\x_k+v_{j,k}) \right\|^2{ \mid \mathcal{F}_k \cup \{\omega_{k,j}\}} \right] \mid \mathcal{F}_k\right]\\
  &\stackrel{\tiny{\hbox{Assumption \ref{ass-1}}}}{\leq}\tfrac{n^2}{\eta^2}\mathbb{E}\left[ \nu^2\mid \mathcal{F}_k\right] = \tfrac{n^2\nu^2}{\eta^2}, \qquad \hbox{a.s.}
\end{align*}
\endgroup
From $\mathbb{E}[e_{j,k}\mid \mathcal{F}_k] =0$ a.s., Assumption \ref{assum:random_vars}, and Definition \ref{def:stoch_errors}, 
\begin{align*}
&\mathbb{E}[\|e_k\|^2\mid \mathcal{F}_k] =\tfrac{1}{N_k^2} \mathbb{E}\left[\left\|\textstyle\sum_{j=1}^{N_k}e_{j,k}\right\|^2\mid \mathcal{F}_k\right]\\
&=\tfrac{1}{N_k^2} \mathbb{E}\left[\textstyle\sum_{j=1}^{N_k}\left\|e_{j,k}\right\|^2\mid \mathcal{F}_k\right] = \tfrac{1}{N_k^2}\textstyle\sum_{j=1}^{N_k} \mathbb{E}\left[\left\|e_{j,k}\right\|^2\mid \mathcal{F}_k\right].
\end{align*}
Combining the preceding two relations, we obtain $\mathbb{E}[\|e_k\|^2\mid \mathcal{F}_k] \leq \tfrac{n^2\nu^2}{\eta^2N_k}$ a.s. To show the second inequality, we have
\begin{align*}
 &\mathbb{E}\left[\| 
\theta_{j,k}\|^2\mid \mathcal{F}_k\right] \\
 & =\mathbb{E}\left[\left\|g_{\eta}(\x_k,v_{j,k},\omega_{j,k}) -\nabla f_\eta(\x_k,\omega_{j,k}) \right\|^2 \mid \mathcal{F}_k\right]\\
&\stackrel{\tiny{\hbox{Assumption \ref{assum:random_vars}}}}{=}
\mathbb{E}\left[\mathbb{E}\left[\left\|g_{\eta}(\x_k,v_{j,k},\omega_{j,k}) -\nabla f_\eta(\x_k,\omega_{j,k}) \right\|^2 \right. \right.\\ 
&\left. \left.\mid \mathcal{F}_k\cup\{\omega_{j,k}\}\right] {\mid \mathcal{F}_k}\right]\\
&\stackrel{\tiny{\hbox{Lemma \ref{lemma:props_local_smoothing}\hbox{ (viii)}}}}{\leq}
\mathbb{E}\left[L_0^2n^2  \mid \mathcal{F}_k\right]=L_0^2n^2,\qquad \hbox{a.s.}
\end{align*}  
Similar to the proof of the first bound, we have that $\mathbb{E}[\|\theta_k\|^2\mid \mathcal{F}_k] \leq \tfrac{L_0^2n^2}{N_k}$ a.s. To {derive} the third bound, we {may express} 
\begingroup
\allowdisplaybreaks
\begin{align*}
 & \mathbb{E}\left[\| 
\delta_{j,k}\|^2\mid \mathcal{F}_k\right]  \\
&=\mathbb{E}\left[\left\|b \mathbf{U}_{i_k}g_{\eta}(\x_k,v_{j,k},\omega_{j,k},i_k) -  g_{\eta}(\x_k,v_{j,k},\omega_{j,k}) \right\|^2 \mid \mathcal{F}_k\right]\\
&\stackrel{\tiny{\hbox{Assum. \ref{assum:random_vars}}}}{=}
 \mathbb{E}\left[  \mathbb{E}\left[ \left\|b \mathbf{U}_{i_k}g_{\eta}(\x_k,v_{j,k},\omega_{j,k},i_k) -  g_{\eta}(\x_k,v_{j,k},\omega_{j,k})\right\|^2\right. \right.\\ 
&\left. \left. \mid \mathcal{F}_k\cup\{\omega_{j,k},v_{j,k}\}\right] { \mid \mathcal{F}_k} \right]\\
  &= \mathbb{E}\left[ \sum_{i=1}^b b^{-1}\left\|b \mathbf{U}_{i_k}g_{\eta}(\x_k,v_{j,k},\omega_{j,k},i) -  g_{\eta}(\x_k,v_{j,k},\omega_{j,k})\right\|^2\right. \\ 
&\left. \mid \mathcal{F}_k\right]= \mathbb{E}\left[ \textstyle\sum_{i=1}^b b\left\| \mathbf{U}_{i_k}g_{\eta}(\x_k,v_{j,k},\omega_{j,k},i)\right\|^2 \right. \\
 &\left. +\left\| g_{\eta}(\x_k,v_{j,k},\omega_{j,k})\right\|^2\right.\\
&\left. -2 g_{\eta}(\x_k,v_{j,k},\omega_{j,k})^T\textstyle\sum_{i=1}^b\mathbf{U}_{i_k}g_{\eta}(\x_k,v_{j,k},\omega_{j,k},i)\mid \mathcal{F}_k\right] \\
 &= \mathbb{E}\left[ \textstyle\sum_{i=1}^b b\left\|g_{\eta}(\x_k,v_{j,k},\omega_{j,k},i)\right\|^2 -\left\| g_{\eta}(\x_k,v_{j,k},\omega_{j,k})\right\|^2\right. \\ 
&\left.  \mid \mathcal{F}_k\right]  = (b-1)\mathbb{E}\left[ \left\| g_{\eta}(\x_k,v_{j,k},\omega_{j,k})\right\|^2 \mid \mathcal{F}_k\right].
\end{align*} 
\endgroup
We can also write 
\begingroup
\allowdisplaybreaks
\begin{align*}
 &  \mathbb{E}\left[ \left\| g_{\eta}(\x_k,v_{j,k},\omega_{j,k})\right\|^2 \mid \mathcal{F}_k\right]\\
& = \mathbb{E}\left[ \left\|e_{j,k }+\theta_{j,k }+\nabla f_\eta(\x_k)\right\|^2 \mid \mathcal{F}_k\right] \\
& \leq 3\mathbb{E}\left[ \left\|e_{j,k }\right\|^2  + \left\|\theta_{j,k }\right\|^2+ \left\|\nabla f_\eta(\x_k)\right\|^2 \mid \mathcal{F}_k\right].
\end{align*} 
Note that from Lemma \ref{lemma:props_local_smoothing} we have
\begin{align*}
&\mathbb{E}\left[ \left\|\nabla f_\eta(\x_k)\right\|^2 \mid \mathcal{F}_k\right]\\
 & \overset{\tiny \mbox{Jensen's}}{\leq} {\left(\tfrac{n}{\eta}\right)^2 \left\| \mathbb{E}\left[ f(\x_k+v)\mid \mathcal{F}_k\right] \right\|^2}\leq \left(\tfrac{n \hat f}{\eta}\right)^2, \qquad { \mbox{ a.s.} } \end{align*}
\endgroup
where the second inequality follows from noting that $\sup_{\x \in \Xscr + \eta \mathbb{S}} f(\x) \leq \hat{f}$. From $\mathbb{E}[\delta_{j,k}\mid \mathcal{F}_k] =0$ {a.s.}, Assumption \ref{assum:random_vars}, and the definition of $\delta_k$, we have
\begin{align*}
&\mathbb{E}[\|\delta_k\|^2\mid \mathcal{F}_k] =\tfrac{1}{N_k^2} \mathbb{E}\left[\left\|\textstyle\sum_{j=1}^{N_k}\delta_{j,k}\right\|^2\mid \mathcal{F}_k\right]\\
&=\tfrac{1}{N_k^2} \mathbb{E}\left[\textstyle\sum_{j=1}^{N_k}\left\|\delta_{j,k}\right\|^2\mid \mathcal{F}_k\right] = \tfrac{1}{N_k^2}\textstyle\sum_{j=1}^{N_k} \mathbb{E}\left[\left\|\delta_{j,k}\right\|^2\mid \mathcal{F}_k\right].
\end{align*}
From the preceding four inequalities we obtain 
\begin{align*}
&\mathbb{E}[\|\delta_k\|^2\mid \mathcal{F}_k]\leq 3(1-b) \left(\tfrac{1}{N_k^2}\textstyle\sum_{j=1}^{N_k}\mathbb{E}\left[ \left\|e_{j,k }\right\|^2 \mid \mathcal{F}_k\right] \right. \\ 
&\left. + \tfrac{1}{N_k^2}\textstyle\sum_{j=1}^{N_k}\mathbb{E}\left[ \left\|\theta_{j,k }\right\|^2 \mid \mathcal{F}_k\right]+  \tfrac{1}{N_k^2}\textstyle\sum_{j=1}^{N_k}\left(\tfrac{n \hat f}{\eta}\right)^2\right)\\
&\leq 3(b-1)\left(\tfrac{n^2\nu^2}{\eta^2N_k}+\tfrac{L_0^2n^2}{N_k}+\left(\tfrac{n \hat f}{\eta}\right)^2\tfrac{1}{N_k}\right) \\
&= \tfrac{3n^2(b-1)}{N_k}\left(\tfrac{\nu^2}{\eta^2}+ L_0^2+\left(\tfrac{ \hat f}{\eta}\right)^2 \right),\qquad \hbox{a.s.}
\end{align*}
\end{proof}

\section{Convergence and rate analysis}
{In this section, we analyze the rate and complexity guarantees for the proposed zeroth-order randomized scheme.} The reader may note that in the {proposed} block-coordinate framework, the $i$th block is randomly selected. If this randomly selected block is denoted by $i_k$ at the $k$th epoch, then a projected (inexact and smoothed) gradient step is taken with respect to this block, i.e.
$$ \x^{(i)}_{k+1} := \Pi_{\Xscr_i} \left[\x^{(i)}_k - \gamma g_{\eta,N_k,i}(\x_k) \right], \mbox{ where } i = i_k$$
while $\x^{(j)}_{k+1} = \x^{(j)}_k$ for $j \neq i_k$. The next Lemma shows that this step can be recast as a projected gradient step for the entire vector $\x$ with respect to the set $\Xscr$. Similar results have been proven in~\cite{Dang15,FarzadSetValued18,KaushikYousefianSIOPT2021}. 
\begin{lemma}
Consider Definition \ref{def:stoch_errors} and Algorithm \ref{algorithm:zo_nonconvex}. Then the update rule of $\x_k$ can be compactly characterized as 
\begin{align*}
\x_{k+1} = \Pi_\Xscr\left[\x_k-b^{-1}\gamma\left(\nabla f_\eta(\x_k)+e_k+\theta_k+\delta_k\right)\right].
\end{align*}
\begin{proof}
Invoking the definition of $\mathbf{U}_{i_k} $ and the Cartesian structure of $\Xscr$, the update rule of $\x_k$ in Algorithm \ref{algorithm:zo_nonconvex} can be written as 
\begin{align}\label{eqn:proof_comapct_eq1}
\x_{k+1} = \Pi_\Xscr\left[\x_k-\gamma  \mathbf{U}_{i_k}  g_{\eta,N_k,i_k}(\x_k) \right].
\end{align}
Summing relations \eqref{def:stoch_errors1}, \eqref{def:stoch_errors2}, and \eqref{def:stoch_errors3} we have for all $k\geq 0$ and $j=1,\ldots,N_k$ that 
\begin{align*}
e_{j,k}+\theta_{j,k }+\delta_{j,k} =b \mathbf{U}_{i_k}g_{\eta}(\x_k,v_{j,k},\omega_{j,k},i_k) - \nabla f_\eta(\x_k).
\end{align*}
Summing over $j$ and dividing by $N_k$, we obtain for $k\geq 0$
\begin{align*}
&e_k+\theta_k+\delta_k = b \mathbf{U}_{i_k} \frac{\sum_{j=1}^{N_k}  g_{\eta}(\x_k,v_{j,k},\omega_{j,k},i_k)}{N_k}-\nabla f_\eta(\x_k) \\
&= b \mathbf{U}_{i_k}  g_{\eta,N_k,i_k}(\x_k) -\nabla f_\eta(\x_k).
\end{align*}
From this relation and \eqref{eqn:proof_comapct_eq1}, we obtain the result.
\end{proof}
\end{lemma}

Next, we derive a bound on the residual $\|G_{\eta, b/\gamma}(\x_k)\|^2$.
\begin{lemma}
Let Assumption \ref{ass-1} hold. Consider Definition \eqref{def:stoch_errors}. Suppose $\x_k$ is generated by Algorithm \ref{algorithm:zo_nonconvex} where $\gamma \in (0,\tfrac{b\eta}{nL_0})$ for $\eta>0$. Then for all $k\geq 0$ we have 
     \begin{align}\label{eqn:recursive_ineq_lemma5} \notag
& \left( 1-\tfrac{nL_0\gamma}{b\eta}\right) \tfrac{\gamma}{4b} \|G_{\eta,b/\gamma} (\x_k)\|^2   \leq   f_{\eta}(\x_k)-      f_{\eta}(\x_{k+1}) \\
&+{\left( 1-\tfrac{nL_0\gamma}{2b\eta}\right)} \tfrac{\gamma}{b} \|e_k+\theta_k+\delta_k\|^2.
    \end{align}
\end{lemma}
\begin{proof}
From Lemma \ref{lemma:props_local_smoothing} (iv), the gradient mapping $\nabla f_\eta$ is Lipschitz continuous with the parameter $L_{\eta}\triangleq \tfrac{nL_0}{\eta}$. From the descent lemma, 
\begin{align}\label{eqn:descent_lem_proof_eq1}
&f_{\eta}(\x_{k+1}) \leq f_{\eta}(\x_{k}) + \nabla f_{\eta}(\x_k)^T(\x_{k+1}-\x_k)\notag\\
&+\tfrac{L_{\eta}}{2}\|\x_{k+1}-\x_k\|^2\notag\\
&=f_\eta(\x_k)+ \left(\nabla f_{\eta}(\x_k)+e_k+\theta_k+\delta_k\right)^T(\x_{k+1}-\x_k)\notag\\
&-(e_k+\theta_k+\delta_k)^T(\x_{k+1}-\x_k)+\tfrac{L_{\eta}}{2}\|\x_{k+1}-\x_k\|^2.
\end{align}
{Invoking 
the properties of the Euclidean projection and the Cartesian structure of $\Xscr$,} we have $
(\x_k-b^{-1}\gamma\left(\nabla f_\eta(\x_k)+e_k+\theta_k+\delta_k\right)-\x_{k+1})^T(\x_k-\x_{k+1})\leq 0$. This implies that 
\begin{align}\label{eqn:descent_lem_proof_eq2}
&\left(\nabla f_\eta(\x_k)+e_k+\theta_k+\delta_k\right)^T(\x_{k+1}-x_k)\notag\\
&\leq -\tfrac{b}{\gamma}\|\x_{k+1}-\x_k\|^2.
\end{align}
{In addition, we may also express $\left(e_k+\theta_k+\delta_k\right)^T(\x_{k+1}-\x_k)$ as follows.} 
\begin{align}\label{eqn:descent_lem_proof_eq3}
&-\left(e_k+\theta_k+\delta_k\right)^T(\x_{k+1}-\x_k) \notag\\
&\leq \tfrac{\gamma}{2b}\|e_k+\theta_k+\delta_k\|^2+\tfrac{b}{2\gamma}\|\x_{k+1}-\x_k\|^2.
\end{align}
Combining the inequalities \eqref{eqn:descent_lem_proof_eq1}, \eqref{eqn:descent_lem_proof_eq2}, and \eqref{eqn:descent_lem_proof_eq3} we obtain 
     \begin{align*}
      f_{\eta}(\x_{k+1})  & \leq   f_{\eta}(\x_k)+ \left( -\tfrac{b}{2\gamma}+\tfrac{L_{\eta}}{2}\right) \|\x_{k+1}-\x_k\|^2 \notag\\
&+\tfrac{\gamma}{2b} \|e_k+\theta_k+\delta_k\|^2.
    \end{align*}
From Definition \ref{def:res_maps} we obtain
     \begin{align*}
         & \quad f_{\eta}(\x_{k+1})  \\
         & \leq   f_{\eta}(\x_k)+ \left( -\tfrac{b}{2\gamma}+\tfrac{L_{\eta}}{2}\right)\tfrac{\gamma^2}{b^2} \left\|\tilde{G}_{\eta,b/\gamma}(\x_k, {e_k+\theta_k+\delta_k})\right\|^2\notag\\
&+\tfrac{\gamma}{2b} \|e_k+\theta_k+\delta_k\|^2\\
&=   f_{\eta}(\x_k)- \left( 1-\tfrac{L_{\eta}\gamma}{b}\right)\tfrac{\gamma}{2b} \left\|\tilde{G}_{\eta,b/\gamma}(\x_k, {e_k+\theta_k+\delta_k})\right\|^2\notag\\
&+\tfrac{\gamma}{2b} \|e_k+\theta_k+\delta_k\|^2.
    \end{align*}
Using Lemma \ref{lem:inexact_proj_2} and {by the requirement that} $\gamma< \tfrac{b}{L_{\eta}}$, 
     \begin{align}
      f_{\eta}(\x_{k+1})  & \leq      f_{\eta}(\x_k)- \left(1-\tfrac{L_{\eta}\gamma}{b}\right)\tfrac{\gamma}{4b} \left\|{G}_{\eta,b/\gamma}(\x_k)\right\|^2\notag \\
&+\left(\tfrac{\gamma}{2b}+\left( 1-\tfrac{L_{\eta}\gamma}{b}\right)\tfrac{\gamma}{2b}\right) \|e_k+\theta_k+\delta_k\|^2.\label{cont-rec}
    \end{align}
From the preceding relation, we obtain the result. 
\end{proof}
We now present an almost sure convergence guarantee for the sequence generated by Algorithm~\ref{algorithm:zo_nonconvex} by relying on the Robbins-Siegmund Lemma.
\begin{proposition}[{\bf Asymptotic guarantees for Alg.~\ref{algorithm:zo_nonconvex}}]
    Consider Algorithm \ref{algorithm:zo_nonconvex}. Let Assumptions~\ref{ass-1} and \ref{assum:random_vars} hold and $N_k:=(k+1)^{1+\delta}$ for $k \geq 0$ and $\delta>0$. Then the following hold. 
\noindent (i) $\|G_{\eta,b/\gamma}(\x_k)\| \xrightarrow[k \to \infty]{a.s.} 0.$

\noindent (ii) {Every limit point of $\{\x_k\}$ lies in the set of $2\eta$-Clarke stationary points of \eqref{eqn:prob} in an almost sure sense.} 

\end{proposition}
\begin{proof}
Let $f^*_\eta\triangleq \inf_{x \in \Xscr} f_\eta(\x)$. By taking conditional expectations with respect to $\mathcal{F}_k$ on the both sides of the inequality \eqref{eqn:recursive_ineq_lemma5}, we have 
\begin{align}
& \quad \mathbb{E}\left[(f_{\eta}(\x_{k+1}) -f^*_\eta) \mid \mathcal{F}_k\right] \notag\\ 
& \leq      (f_{\eta}(\x_k)-f^*_\eta)- \left(1-\tfrac{L_{\eta}\gamma}{b}\right)\tfrac{\gamma}{4b} \left\|{G}_{\eta,b/\gamma}(\x_k)\right\|^2\notag \\
&+\left(\tfrac{\gamma}{2b}+\left( 1-\tfrac{L_{\eta}\gamma}{b}\right)\tfrac{\gamma}{2b}\right) \mathbb{E}\left[\|e_k+\theta_k+\delta_k\|^2 \mid \mathcal{F}_k\right] \notag\\
& \leq      (f_{\eta}(\x_k)-f^*_\eta)- \left(1-\tfrac{L_{\eta}\gamma}{b}\right)\tfrac{\gamma}{4b} \left\|{G}_{\eta,b/\gamma}(\x_k)\right\|^2\notag \\
&+\left(\tfrac{\gamma}{2b}+\left( 1-\tfrac{L_{\eta}\gamma}{b}\right)\tfrac{\gamma}{2b}\right) \tfrac{c}{N_k}.\label{cont-rec3}
    \end{align} 
    By the Robbins-Siegmund Lemma, the summability of
    $\|e_k+\theta_k+\delta_k\|^2$ ({which holds by} choice of $N_k$), and the nonnegativity of
    $f_{\eta}(\x_{k}) -f^*_\eta$, we have that
    $\{(f_{\eta}(\x_k)-f^*_\eta)\}$ is convergent a.s. and
    $\sum_{k=1}^{\infty} \left\|{G}_{\eta,b/\gamma}(\x_k)\right\|^2 < \infty$
almost surely. It remains to show that with probability one, $\|G_{\eta,b/\gamma}(\x_k)\| \to 0$ as $k \to \infty$.
We proceed by contradiction. Suppose for $\omega \in \Omega_1 \subset
\Omega$ and $\mu(\Omega_1) > 0$ (i.e. with finite probability),
$\|G_{\eta,b/\gamma}(\x_k)\| \xrightarrow{k \in \Kscr(\omega)} \epsilon(\omega)
> 0$ where $\Kscr(\omega)$ is a random subsequence. Consequently, for every
$\omega \in \Omega_1$ and $\tilde{\delta} > 0$, there exists $K(\omega)$ such
that $k \geq K(\omega)$, $\| G_{\eta,b/\gamma}(\x_k) \| \geq
\tfrac{\epsilon(\omega)}{2}$.    Consequently, we have that $\sum_{k \to
\infty} \|G_{\eta,b/\gamma}(\x_k)\|^2 \geq \sum_{k \in \Kscr(\omega)}
\|G_{\eta,b/\gamma} (\x_k)\|^2 \geq \sum_{k \in \Kscr(\omega), k \geq
K(\omega)} \|G_{\eta,b/\gamma} (\x_k)\|^2 = \infty$ with finite probability.
But this leads to a contradiction, implying that $\|G_{\eta,b/\gamma}(\x_k)\|^2
\xrightarrow[k \to \infty]{a.s.} 0.$ 

\smallskip

\noindent {(ii) Recall from Proposition \ref{prop_equiv} that if $\x$ satisfies $G_{\eta,b/\gamma}(\x) = 0$, it is a $2\eta$-Clarke stationary point of \eqref{eqn:prob}, i.e. $0 \in \partial_{2\eta} f(\x)+\Nscr_{\Xscr}(\x).$ Since almost every limit point of $\{\x_k\}$ satisfies $G_{\eta,b/\gamma}(\x) = 0$, the result follows.}  
\end{proof}
 
{We now conclude this section with a formal rate statement and complexity guarantees in terms of projection steps on $\Xscr_i$ as well as sampled function evaluations.}
\begin{theorem}[{\bf Rate and complexity {statements for Alg.}~\ref{algorithm:zo_nonconvex}}]\label{thm:vr-rb-zo}\em
    Consider Algorithm \ref{algorithm:zo_nonconvex}. Let Assumptions~\ref{ass-1} and \ref{assum:random_vars} hold {and} ${N_k:=\lceil 1+ \tfrac{k+1}{\eta^a}\rceil}$ for $k \geq 0$ and for some {$a \geq 0$}. 
    
\noindent {\bf{(i)}} For $\gamma<\frac{b\eta}{n L_0}$ and all $K> \tfrac{2}{1-\lambda}$ with $\ell\triangleq  \lceil \lambda K\rceil$ we have
\begin{align*}
& \mathbb{E}\left[ \|G_{\eta,b/\gamma} (\x_R)\|^2\right] \leq \left(\left( 1-\tfrac{nL_0\gamma}{b\eta}\right) \tfrac{\gamma}{4b}(1-\lambda)K \right)^{-1}\\
& \left(\mathbb{E}\left[f(\x_{\ell})\right] -f^* +2L_0\eta+ 3n^2\left((3b-2)\left(\nu^2+ L_0^2 {\eta^2}\right)\right.\right.\\
&\left.\left.+3(b-1){{\hat f}^2} \right)(0.5-\ln(\lambda)) {\eta^{2-a}}\right).
\end{align*}

\noindent {\bf{(ii)}} Suppose $\gamma= \tfrac{b\eta}{2nL_0}$. Let $\epsilon>0$ be an arbitrary scalar and $K_{\epsilon}$ be such that $ \mathbb{E}\left[ \|G_{\eta,b/\gamma} (\x_R)\|\right]   \leq \epsilon$. Then:

    \noindent {{(ii-1)}} The total number of projection steps on component sets is $K_{\epsilon}=\mathcal{O}\left( {\eta^{-1+[a-2]_+}}\epsilon^{-2}\right)$.

\noindent {{(ii-2)}} The total sample complexity is $\mathcal{O}\left({\eta^{-a-2+2[a-2]_+}}\epsilon^{-4} \right)$.
\end{theorem}
\begin{proof}
\noindent {\bf (i)} Consider the inequality of Lemma \ref{lem:inexact_proj_2}.  Summing from $k =\ell, \ldots, K-1$ where $\ell\triangleq  \lceil \lambda K\rceil$ we have
 $ \left( 1-\tfrac{nL_0\gamma}{b\eta}\right) \tfrac{\gamma}{4b}\sum_{{k=\ell}}^{K-1} \|G_{\eta,b/\gamma} (\x_k)\|^2   \leq   f_{\eta}(\x_\ell)-      f_{\eta}(\x_{K}) 
  +{\left( 1-\tfrac{nL_0\gamma}{2b\eta}\right)} \tfrac{\gamma}{b} \sum_{{k=\ell}}^{K-1}\|e_k+\theta_k+\delta_k\|^2.$
Taking expectations {on} both sides, we obtain
\begin{align*}
    & \quad \left( 1-\tfrac{nL_0\gamma}{b\eta}\right) \tfrac{\gamma}{4b}(K-\ell)\mathbb{E}\left[ \|G_{\eta,b/\gamma} (\x_R)\|^2\right]\leq \\
    &  {\left( 1-\tfrac{nL_0\gamma}{2b\eta}\right)} \tfrac{\gamma}{b}\sum_{k=\ell}^{K-1} \mathbb{E}\left[\|e_k+\theta_k+\delta_k\|^2\right] + \mathbb{E}\left[f_{\eta}(\x_{\ell})\right] - {f^*_{\eta}}.
\end{align*}
{The last term on the right can be expressed as} 
\begin{align*}
& \mathbb{E}\left[f_{\eta}(\x_{\ell})\right] - {f^*_{\eta}} =   \mathbb{E}\left[ f(\x_{\ell}) + f_{\eta}(\x_{\ell})-f(\x_{\ell})\right]-f^*_{\eta} +f^*\\
 &-f^*= \mathbb{E}\left[{f(\x_{\ell})}\right]-f^*+ \mathbb{E}\left[\left | f_{\eta}(\x_{\ell})-f(\x_{\ell})\right|\right] + \left| f^*-f^*_\eta\right| \\
&\stackrel{\tiny{\hbox{Lemma \ref{lemma:props_local_smoothing}(iii)}}}{\leq}
 \mathbb{E}\left[f(\x_{\ell})\right] -f^* +2L_0\eta.
\end{align*}
From Lemma \ref{lem:stoch_error_var} (ii), we obtain
\begin{align*}
&\mathbb{E}\left[\|e_k+\theta_k+\delta_k\|^2\right] \leq 3\mathbb{E}\left[\mathbb{E}\left[\|e_k\|^2+ \|\theta_k\|^2+\|\delta_k\|^2\mid \mathcal{F}_k\right]\right] \\
& \leq \tfrac{{3}n^2\nu^2}{\eta^2N_k}+\tfrac{{3}L_0^2n^2}{N_k}+\tfrac{{9}n^2(b-1)}{N_k}\left(\tfrac{\nu^2}{\eta^2}+ L_0^2+\left(\tfrac{ \hat f}{\eta}\right)^2 \right)\\
& \leq \tfrac{3n^2}{N_k}\left(\tfrac{(3b-2)\nu^2}{\eta^2}+ (3b-2)L_0^2+3(b-1)\left(\tfrac{ \hat f}{\eta}\right)^2 \right).
\end{align*}
From the preceding relations we obtain
\begin{align*}
 &\mathbb{E}\left[ \|G_{\eta,b/\gamma} (\x_R)\|^2\right] \leq 
{\left(\left( 1-\tfrac{nL_0\gamma}{b\eta}\right) \tfrac{\gamma}{4b}(K-\ell)\right)}^{-1}{\times}\\
 &\left( \mathbb{E}\left[f(\x_{\ell})\right] -f^* +2L_0\eta+ \left((3b-2)\nu^2+ (3b-2)L_0^2 {\eta^2}\right.\right.\\
 &\left.\left.+3(b-1){(\hat f)^2} \right)\sum_{k=\ell}^{K-1} \tfrac{3n^2}{{\eta^2}N_k}\right) \\
& \leq {\left(\left( 1-\tfrac{nL_0\gamma}{b\eta}\right) \tfrac{\gamma}{4b}(K-\ell)\right)}^{-1}{\times}\\
 &\left( \mathbb{E}\left[f(\x_{\ell})\right] -f^* +2L_0\eta+ \left((3b-2)\nu^2+ (3b-2)L_0^2 {\eta^2}\right.\right.\\
 &\left.\left.+3(b-1){(\hat f)^2} \right)\sum_{k=\ell}^{K-1} \tfrac{3n^2}{{\eta^{2-a}}{k+1}}\right),  
\end{align*}
{where the last inequality is a consequence of 
$N_k:={\lceil 1+\tfrac{k+1}{\eta^a} \rceil \geq \tfrac{k+1}{\eta^a}}$}. {Recall} that $K> {\tfrac{2}{1-\lambda}}$ implies $\ell \leq K-1$ and  
 $ \textstyle\sum_{k={\ell}}^{K-1}\tfrac{1}{k+1}
 \leq 0.5 +\ln\left(\tfrac{N}{\lambda N+1}\right)\leq 0.5-\ln(\lambda).$
 {Further}, $K-\ell \geq K-\lambda K=(1-\lambda)K$ {implying that} 
\begin{align*}
 &\mathbb{E}\left[ \|G_{\eta,b/\gamma} (\x_R)\|^2\right] \leq \left(\left( 1-\tfrac{nL_0\gamma}{b\eta}\right) \tfrac{\gamma}{4b}(1-\lambda)K\right)^{-1}{\times}\\
 &\left(\mathbb{E}\left[f(\x_{\ell})\right] -f^* +2L_0\eta+ 3n^2\left((3b-2)\nu^2+ (3b-2)L_0^2 {\eta^2}\right. \right.\\
 &\left.\left.+3(b-1) {\hat f}^2 \right)(0.5-\ln(\lambda)){\eta^{a-2}}\right) .
\end{align*}
\noindent {\bf (ii)} To show (ii-1), using the relation in part (i) and substituting $\gamma:= \tfrac{b\eta}{2nL_0}$ we obtain
\begin{align*}
& \mathbb{E}\left[ \|G_{\eta,b/\gamma} (\x_R)\|^2\right] \leq (16nL_0)\left(\eta(1-\lambda)K\right)^{-1} {\times} \\
& \left(\mathbb{E}\left[f(\x_{\ell})\right] -f^* +2L_0\eta+ 3n^2\left((3b-2)\nu^2+ (3b-2)L_0^2 {\eta^2}\right. \right.\\
&\left.\left.+3(b-1) {\hat f}^2 \right)(0.5-\ln(\lambda)){\eta^{a-2}}\right).
\end{align*}
From Jensen's inequality, {it follows} that $ \mathbb{E}\left[ \|G_{\eta,1/\gamma} (\x_R)\|\right] \leq \sqrt{\mathcal{O}\left({\eta^{-1+[a-2]_+}K^{-1}}\right)}$ and thus, we have $K_{\epsilon}=\mathcal{O}\left({\eta^{-1+[a-2]_+}}\epsilon^{-2}\right)$. Next, we show (ii-2). The total sample complexity of upper-level is as follows.
\begin{align*}
    \textstyle\sum_{k=0}^{K_\epsilon} N_k & = \textstyle\sum_{k=0}^{K_\epsilon} \lceil 1+ \tfrac{(k+1)}{\eta^a}\rceil 
        \leq \mathcal{O}(K_{\epsilon}) +  \mathcal{O} (\tfrac{K_\epsilon^2}{\eta^{a}})\\
        & \leq \mathcal{O}\left({\eta^{-a-2+2[a-2]_+}} \epsilon^{-4}\right).
\end{align*}
\end{proof}
\begin{remark}
        (i) For $0\leq a<2$, {we attain the best iteration complexity of $\mathcal{O}\left( \eta^{-1}\epsilon^{-2}\right)$ and sample complexity of $\mathcal{O}\left( \eta^{-2}\epsilon^{-4}\right)$ in terms of dependence of $\eta$ when $a = 0$.} (ii) When $a\geq 2$, these statements change to $\mathcal{O}\left( \eta^{a-3}\epsilon^{-2}\right)$ and $\mathcal{O}\left( \eta^{a-6}\epsilon^{-4}\right)$, respectively. Within this range, for $a:=6$ we obtain a sample complexity of $\mathcal{O}\left(\epsilon^{-4}\right)$ that is invariant in terms of $\eta$, {but  a small $\eta$ leads to a large batch size making the implementation of the scheme less appealing}. Exploring these trade-offs remains a future direction to our research. (iii) Our results are comparable to those obtained in~\cite{zhang20complexity}, which uses neither smoothing nor a zeroth-order framework. 
\end{remark}

\section{Concluding remarks}
While a significant amount of prior research has analyzed nonsmooth and
nonconvex optimization problems, much of this effort has relied on either the
imposition of structural assumptions on the problem or required weak convexity,
rather than general nonconvexity. Little research, if any, is available in
stochastic regimes to contend with general nonconvex and nonsmooth optimization
problems. To this end, we develop a randomized smoothing framework which
allows for claiming that a stationary point of the $\eta$-smoothed problem is a
$2\eta$-stationary point for the original problem in the Clarke sense. By
utilizing a suitable residual function that provides a metric for stationarity
for the smoothed problem, we present a  zeroth-order framework reliant on
utilizing sampled function evaluations implemented in a block-structured
regime. In this setting, we make two sets of contributions for the sequence
generated by the proposed scheme. (i) The residual function of the smoothed
problem tends to zero almost surely along the generated sequence; (ii) To
compute an $\x$ that ensures that the expected {norm of} {the residual of the $\eta$-smoothed problem} is within
$\epsilon$, we proceed to show that no more than $\mathcal{O}(\eta^{-1} \epsilon^{-2})$
projection steps and $\mathcal{O}\left(\eta^{-2}\epsilon^{-4}\right)$ function
evaluations are required.

\bibliographystyle{plain}
\bibliography{demobib_v1,wsc11-v03a,ref_paIRIG_v01_fy}

\end{document}